\documentclass[12pt,a4paper]{amsart}

\newcommand{\IR}{\mathbb{R}}

\newcommand{\IN}{\mathbb{N}}

\newtheorem{thm}{Theorem}

\newtheorem{lem}[thm]{Lemma}
\newtheorem{prop}[thm]{Proposition}
\theoremstyle{definition}

\theoremstyle{remark}

\newtheorem{conj}[thm]{Conjecture}

\begin{document}

\title{ A   uniformly bounded complete\\ Euclidean system }


\author{ K.S. Kazarian}

\address{ Departamento de Matem\'aticas, Facultad de Ciencias, Mod. 17\\
		Universidad Aut\'{o}noma de Madrid, Madrid,
		28049, SPAIN }
\email{ kazaros.kazarian@uam.es}


\begin{abstract}{  A uniformly bounded complete orthonormal system of functions $\Theta =\{
\theta_n\}_{n=1}^{\infty},$  $ \|\theta_n\|_{L^\infty_{[0,1]} } \leq M $  is
constructed such that $\sum_{n=1}^{\infty} a_{n}\theta_{n}$
converges almost everywhere on $[0,1]$ if $\{ a_n\}_{n=1}^{\infty}
\in \, l^2$ and $\sum_{n=1}^{\infty} a_{n}\theta_{n}$ diverges a. e.
for any $\{ a_n\}_{n=1}^{\infty} \not\in \, l^2$.

Thus Menshov's theorem  on
the representation of measurable,  almost everywhere finite,
functions by almost everywhere convergent trigonometric series can not be extended to the class of uniformly bounded complete orthonormal systems. }
\end{abstract}

\maketitle

\

\

\section{ Introduction}
The history of the study of pointwise convergence of the expansions by general orthogonal complete systems goes back to the beginning of the twentieth century. Among others one can recall the example constructed by H.Steinhaus \cite{S:20},\cite{KZ:35} of a complete orthonormal system such that the expansion by the system of an integrable function diverges almost everywhere.
  An orthonormal system (ONS) $\{ \varphi_n\}_{n=1}^{\infty}$
of functions defined on a closed interval $[a,b] $ is called a convergence
system if $\sum_{n=1}^{\infty} a_{n}\varphi_{n}$ converges almost everywhere (a.e.) for
any $\{ a_n\}_{n=1}^{\infty} \in \, l^2$.
The history of studies in convergence and divergence of orthogonal series has a long story (see \cite{KZ:35},\cite{A:61},\cite{Ul:64}).
P.L. Ul'yanov (see \cite{Ul:64}) posed various problems in this area which stimulated research in this area. Particularly, B.S. Kashin \cite{Ka:76} responding to a problem posed in \cite{Ul:64} prove that
there exists a complete ONS $\{
\phi_n\}_{n=1}^{\infty}  $ of
functions defined on $[0,1] $ which is a convergence system and for any $\{ a_n\}_{n=1}^{\infty} \not\in \, l^2$ the series $\sum_{n=1}^{\infty} a_{n}\varphi_{n}$ diverges on some set of positive measure.
An ONS $\{ \varphi_n\}_{n=1}^{\infty}$
of functions defined on a closed interval $[a,b] $ is called a
  divergence system if the series
$\sum_{n=1}^{\infty} a_{n}\varphi_{n}$ diverges
a.e. on  $[a,b]$ for any $\{ a_n\}_{n=1}^{\infty} \not\in \, l^2$.

Another problem posed in (\cite{Ul:64},p.695) asks if there exists a complete ONS which is simultaneously a convergence and a divergence system.
B.S. Kashin indicated in \cite{Ka:76} that this problem remains open. An affirmative answer was given by the author \cite{K:06}.\newline
An ONS $\{ \varphi_n\}_{n=1}^{\infty}$ is called a  Euclidean system if it is both a convergence and a divergence system.

  A system of functions $\{ \phi_n\}_{n=1}^{\infty}$ defined on $[a,b] $ is called uniformly bounded if there exists $M>0$ such that
\[
\|\phi_n\|_{L^\infty_{[a,b]}} \leq M \quad\mbox{for all}\quad n\in
\IN.
\]
In the present paper we construct a  uniformly bounded complete Euclidean system.
  We prove the following
\begin{thm}\label{t:1}
 For any $M > 1+ \sqrt{2}$ there exists a complete  Euclidean system $\Theta = \{\theta_{n}\}_{n=1}^{\infty}$ in $L^2_{[0,1]}$  such that
 \begin{equation}\label{ub:1}
\| \theta_n\|_{L^\infty_{[0,1]}} \leq M \quad\mbox{for all}\quad n\in
\IN.
\end{equation}
\end{thm}

As an immediate corollary we obtain that  Menshov's theorem \cite{M:41} on
the representation of measurable,  almost everywhere finite,
functions by almost everywhere convergent trigonometric series can not be extended to the class of uniformly bounded complete orthonormal systems. Moreover, the system $ \Theta$ is not a
representation system for the classes $L_{[0,1]}^r, 0\leq r<2$ if
we want to represent the functions from those classes by a series
 which converges  pointwise on sets of
positive measure, even if those sets depend on the function. Other corollaries of Theorem \ref{t:1} can be find in \cite{K:04}.

In the theory of general orthogonal series uniformly bounded ONS are one of the main objects that have been studied systematically. In the survey article \cite{Ul:64}   Ul'yanov posed the problem of the existence of  complete uniformly bounded convergence
system. It was motivated by the known open problem about the a.e. convergence of the Fourier series. Giving an answer to Ul'yanovs problem Olevskii \cite{O:66} constructed such a convergence system. The idea of the construction can be described as follows. At first step construct a complete ONS of bounded functions which can be divided into two convergence systems such that  the second one is uniformly bounded. In the construction it is the Rademacher system. Afterwards any element of the first convergence system is ``dissolved" by the Rademacher functions in such a way that the resulting functions are uniformly bounded. This process is performed by special orthogonal matrices. Those special matrices afterwards were used for various constructions. It created among some experts an impression that those matrices are remarkable by themselves. Probably this believe do not permit to some group of experts to admit that those matrices were known in applied mathematics much earlier as Haar matrices.
  We  will return to the Haar matrices later on. The important novelty in Olevskii's construction was
   the idea of dissolution by orthogonal transformations of ``bad" elements of a complete ONS by ``good" ones. Of course, at first one should be able to obtain a CONS for which such a construction can be applied. It should be mentioned that the idea of sticking together orthogonal functions by some orthogonal transformations was applied earlier by V. Kostitzin \cite{Ko:13}. We should also mention
 L.Carleson's \cite{C:64} famous article where it was proved that the trigonometric system is a convergence system.

Let us  explain what we understand by saying that some function is ``dissolved" by the Rademacher functions. Moreover, it is done in a such  way that the resulting functions are uniformly bounded.
Let $M > 1+ \sqrt{2}$ and
 suppose we have two functions $\phi_0, \phi_1$ such that
 $$
\|\phi_0\|_{\infty} =\lambda \geq M\qquad \mbox{
and}\quad \|\phi_1\|_{\infty}\leq 1,
$$
then the orthogonal
transformation of those functions by the matrix
\begin{equation}\label{m:1}
A_0 = \left(
\begin{array}{cc} { \frac1{\sqrt{2}}} & { \frac1{\sqrt{2}}}
\\
\frac1{\sqrt{2}} & -\frac1{\sqrt{2}}
\end{array} \right)
\end{equation}
will give two functions $\psi_i = \frac1{\sqrt{2}}( \phi_0
+(-1)^{i-1} \phi_1)\, (i=1,2)$ such that
\[
\|\psi_i\|_{\infty} \leq 2^{-1/2}(1+\lambda) < \lambda \qquad
i=1,2.
\]
Repeating the process for the pairs $(\psi_1, \phi_3)$, $(\psi_2,
\phi_4),$ where
$$\|\phi_i\|_{\infty}\leq 1\, (i=3,4)$$
and, if
necessary,   for the obtained new functions one can easily check
that  on some step the obtained functions will have
$L^\infty_{[0,1]}-$norm less than $M.$ If we have an
infinite subsystem of functions uniformly bounded by $C>0$ then
the same process will give functions with $L^\infty_{[0,1]}-$norm
less than $C\cdot M.$
It seems that the solution of the following conjecture needs some new ideas.
 \begin{conj}
 There is no  complete Euclidean system  $\{ \varphi_n\}_{n=1}^{\infty}$ in $L^2_{[0,1]}$ such that
 \[
 |\varphi_n(t)| = 1 \qquad \mbox{a.e. on}\, [0,1]\qquad \mbox{for any}\, n\in \mathbb{N}.
 \]
\end{conj}

 The present paper consists of $5$ sections. In  Section \ref{dar:1} are given
definitions and auxiliary results many of which can be consulted in the previous papers \cite{K:04} and \cite{K:06}  of the
author. In Section \ref{sim:1} we repeat the construction of two auxiliary complete orthonormal systems from \cite{K:06}, where it was proved that those systems are convergence
 systems. At the end of Section \ref{sim:1} a Euclidean system
 $\{\Upsilon_{k}(x)\}_{k=1}^{\infty}$
is constructed such that by adding a subsystem of the  Rademacher
functions to $\{\Upsilon_{k}(x)\}_{k=1}^{\infty}$ we will obtain a
complete orthonormal system. In Section \ref{ubos:1} one can find
the proof that $\{\Upsilon_{k}(x)\}_{k=1}^{\infty}$ is a system of
divergence. Moreover, we prove an essentially stronger result (see
Theorem \ref{t:12}) which is fundamental for the proof of Theorem
\ref{t:1}. That the system $\Theta$ is a convergence system follows immediately
from Proposition \ref{pro:1} and from the construction of the
auxiliary system $\{\chi_{k}(x)\}_{k=1}^{\infty}$.

\

\section{  Definitions and auxiliary results }\label{dar:1}

We repeat some notations from  \cite{K:04}.
   For  $[a,b]\subset \IR$ and $k\in \IN$ let
\[
 {\mathcal E}^k_{[a,b]} =\left\{ f: f(x) = a_i\, \mbox {if }\,
x\in \left( a+ (i-1){\frac{b-a}{2^k}}, a+
i{\frac{b-a}{2^k}}\right) \right\},
\]
 where $ 1\leq i\leq 2^k$ and
the inner product is defined in the same way as in $L^2_{[a,b]}.$
Further we ignore the values at the points of
discontinuity of functions from ${\mathcal E}^k_{[a,b]} $.
 In the paper we will use also the following notation:
$$ {\mathcal
E}^{j}_{[a,b]} = {\mathcal
E}^{k}_{[a,b]}\stackrel{L^2_{[a,b]}}{\oplus }{\mathcal
E}^{k,j}_{[a,b]},$$ where $j>k.$ In what follows we will denote by
$I_E({ x})$ the characteristic function of a measurable set $E$.

One of the main tools in our
construction will be the Menshov functions $M_k,k\geq 3$ which are
odd $2-$periodic functions defined  on the real line $\mathbb{R}$
and $M_k\in {\mathcal E}^{k+1}_{[-1,1]}.$
 For any natural $k\geq 3$ we define $M_k\in {\mathcal
E}^{k+1}_{[-1,1]}$  to be an odd $2-$periodic function  on the
real line $ {\IR}$ satisfying to the following equations:
\[
M_k(x)=\left\{\begin{array}{ll} {\frac1{8i}}\cdot2^{\frac{k}2},\quad  &{\mbox {if }}\quad x\in (\frac{i-1}
{2^{k}},\frac{i} {2^{k}})\quad 1\leq |i|\leq 2^k -1,\\ \qquad
0, &\mbox{if }\quad x\in (-{\frac {1} {2^{k}}},0)\cup ( 1-{
\frac{1} {2^{k}}},1).
\end{array}\right.
\]
Denote $ M_{k,i}(x)= M_k(x-i\cdot
2^{-k})$ where $ i\in  \IN. $\newline
The following lemma was proved in \cite{K:04}.

\begin{lem}\label{lem:1} For any $k\in {\bf {N}} $ there exist an orthonormal system
 $\{ f_k^i\}_{i=0}^{2^k-1}$ in $L^2_{[-2,2]}$ such that
 \[
f_k^i(x)=\left\{\begin{array}{ll}
   M_{k,i}(x)
,\quad &{\mbox {if }}\quad x\in [-1,1];\\
 0, \qquad &{\mbox {if }}\quad x\in [-2,-1);
\end{array}\right.
\]
\begin{equation}\label{eq:4}
\int_{-2}^{2} f_k^i(x) dx = 0\quad \mbox{for all}\quad 0\leq i
\leq 2^k-1,
\end{equation}
and $f_k^i|_{[1,2]}\in {\mathcal
E}^{k+1}_{[1,2]}.$
\end{lem}

The Haar functions are defined in the following way: for all
$t\in[0,1]$ we will take $ h_1(t)=1 $ and for $k=0,1,2,\dots$;
$j=1,2,\dots,2^k$, let
\[
 h^{(k)}_j(t)=\left\{\begin{array}{ll}\,
\, \,  2^{\frac{k}{2}},\quad &{\mbox {if }}\quad
\frac{2j-2}{2^{k+1}}<t< \frac{2j-1}{2^{k+1}} ;\\
-2^{\frac{k}{2}},\quad\, &\mbox{if }\quad \frac{2j-1}{2^{k+1}}<t<
\frac{2j}{2^{k+1}} ;\\  \quad 0,  &{\mbox {otherwise.}}
\end{array}\right.
\]
If $n=2^k+j$ we denote
$h_{n}=h^{(k)}_j.$
The closure of the support of a the Haar function $h_{n}$ will be
denoted by $\Delta_n $ or by $\Delta^{(k)}_j.$ It can be easily checked that for any $k\in {\IN}$ the Haar
functions $\{ h_n\}_{n=1}^{2^k}$ constitute an orthonormal basis
in the space $ {\mathcal E}^k_{[0,1]}.$

 Recall orthogonal Haar matrices $H_{k}, k\in \mathbb{N}$, that arise from the Haar system. For any  $k\in \mathbb{N }$ we take the midpoints $x_{j}^{(k)} = (j - \frac{1}{2})2^{-k}$ of the intervals $\Delta^{(k)}_j =(\frac{j-1}{2^{k}}, \frac{j}{2^{k}})  $ and set
\begin{equation*}\label{eq:hm}
H_{k} = \left( a_{ij}^{(k)}\right) \quad 1\leq i,j \leq 2^{k}, \, k\in \mathbb{N},
\end{equation*}
where
\begin{equation}\label{eq:hm1}
a_{ij}^{(k)} = 2^{-\frac{k}{2}} h_{i}(x_{j}^{(k)}).
\end{equation}

The Rademacher system $\{r_n(t)\}_{n=0}^\infty $ is an orthonormal
system of functions defined on the closed interval $[0,1].$  It is
convenient for us to consider the Rademacher functions defined on
the real line:
\[
 r_n(t) = \mbox{sgn}\,(\sin 2^{n+1} \pi t) \qquad \qquad
t\in {\IR}, n= 0,1,\dots .
\]
 For our construction it is useful
to note that
\begin{equation}\label{rad:1}
r_n(\cdot)\in {\mathcal E}^{n,n+1}_{[0,1]},\, h_1^{n}(\cdot)\in
{\mathcal E}^{n,n+1}_{[0,1]} \quad \mbox{for all}\quad n\in {\IN}.
\end{equation}

 We also
need the following two lemmas from \cite{K:06}(see Lemma 2.9 and Lemma 2.7).
\begin{lem}\label{le:11} Let $\{ f_k\}_{k=1}^{N+1}$ be a
collection of  functions on $[0,1]$ such that for any $\{
a_k\}_{k=1}^{N+1} \subset \IR$
 $$
 \int_{[0,1]} \sup_{m} \bigg|\sum_{k=1}^m a_kf_k({x})\bigg|^2 dx  \leq C
 \sum_{k=1}^{N+1}a_k^2
 $$
 for some $C>0.$  Then for any $\{
b_k\}_{k=1}^{N+1} \subset \IR$
 $$
 \int_{[0,1]} \max_{1\leq m\leq N} \bigg|\sum_{k=1}^m b_k\widetilde{f}_k({x})\bigg|^2 dx  \leq 14C
 \sum_{k=1}^{N}b_k^2,
 $$
 where
 $$
\widetilde{f}_{j}(x) = -\delta_N\sum_{i=1}^{N}f_{ i+1}(x) +
f_{j+1}(x) +  \frac1{\sqrt{N+1}}f_1(x)\qquad (1\leq j\leq N),
 $$
 and $\delta_N={ \frac1{N}}(1+ {
\frac1{\sqrt{N+1}}}). $
\end{lem}

We give the proof of the following lemma because the value of the constant $C_{p}$ is adjusted. Of course the exact value of the constant is
not important for the proof of the main result but the lemma may be interesting by itself.
\begin{lem}\label{l:61} Let $f\in L^p_{[0,1]}, 1\leq p<\infty$ and
$$\Lambda_{f}(t):= |\{x\in [0,1]: |f(x)|> t\}|.$$ Define
\begin{equation*}\label{eq:61}
f^{\natural}_{p}(x) = 2^{k/p}f\bigg(2^{k}(x-2^{-k})\bigg)
\quad\mbox{when}\quad x\in (2^{-k},2^{-k+1}],\quad k\in\IN.
\end{equation*}
Then
\[\Lambda_{f^{\natural}_{p}}(t) \leq
C_{p}t^{-p}\|f\|_{p}^{p},\quad \mbox{where}\quad C_{p} =
\frac{2}{p(2^{1/p}-1)}.
\]
\end{lem}
\begin{proof}
The proof is straightforward. We have that
\begin{eqnarray*}
\Lambda_{f^{\natural}_{p}}(t) =&&\sum_{k=1}^{\infty}|\{x\in
(2^{-k},2^{-k+1}]: |f^{\natural}(x)_{p}|> t\}|\\ =&&
\sum_{k=1}^{\infty}2^{-k}|\{x\in [0,1]: |f(x)|> 2^{-k/p}t\}| =
\sum_{k=1}^{\infty}2^{-k}\Lambda_{f}(2^{-k/p}t).
\end{eqnarray*}
Afterwards, we write
\[
{2^{-k}\Lambda_{f}(2^{-k/p}t)} =
t^{-p}(2^{-k/p}t)^{p-1}\Lambda_{f}(2^{-k/p}t)(2^{-k/p}-
2^{-\frac{k+1}p})t\cdot \frac1{1-2^{-1/p}}.
\]
Observe that
\[
(2^{-(k+1)/p}t)^{p-1}\Lambda_{f}(2^{-k/p}t) \leq x^{p-1} \Lambda_{f}(x) \qquad \mbox{if}\quad x\in (2^{-\frac{k+1}p}t, 2^{-\frac{k}p}t)
\]
Hence,
\[
(2^{-(k+1)/p}t)^{p-1}\Lambda_{f}(2^{-k/p}t)(2^{-k/p}-
2^{-\frac{k+1}p})t \leq \int_{2^{-\frac{k+1}p}t}^{2^{-\frac{k}p}t} x^{p-1} \Lambda_{f}(x)dx
\]
and  the proof is easily finished
  recalling the formula \newline
   $\|f\|_{p}^{p} = p\int_{0}^{\infty}
t^{p-1}\Lambda_{f}(x)dx.$
\end{proof}

Recall that a system of functions $\{ f_k\}_{k=1}^{\infty}$ defined
 on $[0,1]$ is called an $S_p-$system $(2<p<\infty)$ if for any $\{
a_k\}_{k=1}^{N} \subset \IR$
 $$
 \bigg\|\sum_{k=1}^m a_kf_k(\cdot)\bigg\|_p   \leq C_p
 \left(\sum_{k=1}^{N}a_k^2\right)^{1/2}
 $$
 for some $C_p>0.$ The following result is well known (see
 \cite{G:72}).
 \begin{prop}\label{pro:1} Let $\{ f_k\}_{k=1}^{\infty}$ be an $S_p-$system
 $(2<p<\infty).$
 Then for any $\{
a_k\}_{k=1}^{\infty} \in l^2$ $$
 \bigg\|\sup_{m}\sum_{k=1}^m a_kf_k(\cdot)\bigg\|_p   \leq C'_p
 \left(\sum_{k=1}^{\infty}a_k^2\right)^{1/2}
 $$
 for some $C'_p>0.$
 \end{prop}
The Khintchine inequalities (see \cite{K:93}) show that the Rademacher system is an
$S_p-$system $(2<p<\infty).$ The definition of a set(system) of independent functions can be consulted in \cite{K:93}, \cite{G:72} and others.

\

\section{  Construction  of a  CONS of bounded functions}\label{sim:1}

For the completeness of the exposition we repeat the construction of two auxiliary complete orthonormal systems from \cite{K:06}.
For the convenience of the reader we will  maintain some notations
of the cited paper.

\subsection{ Construction of the first auxiliary CONS} We
suppose that the  orthonormal set of functions
 $\{ f_k^i\}_{i=0}^{2^k-1}, k\in {\IN}$ defined in Lemma \ref{lem:1}
are extended periodically with period $4$ on the whole line and
 define
 $$
g_1^i(x)=\left\{\begin{array}{ll}
 2  f_2^i(8x-2 )
,\quad &{\mbox {if }}\quad x\in [0,\frac12];\\
 r_{9+i}(x)I_{(\frac12,1]}, \qquad &{\mbox {if }}\quad x\in
 (\frac12,1],
\end{array}\right.
$$
for $0\leq i\leq 2^2-1.$  It is easy to check that the functions
$\{{g}_{1}^i(x)\} _{i=0}^{3}$ are orthonormal in  the space
$L^2_{[0,1]}.$ Let $k_1$ be the smallest natural number such that
\begin{equation*}
g_1^i \in {\mathcal E}^{k_1}_{[0,1]} \quad 0\leq i\leq 2^2-1.
\end{equation*}
We take a set of orthonormal functions
\begin{equation*}\label{psi:1}
\psi_\nu \in {\mathcal E}^{k_1+2}_{[0,1]}, 1\leq \nu \leq
2^{k_1+2} -6 = m_1,
\end{equation*}
 $\psi_1(t) =1\quad \mbox{if} \quad t\in [0,1], $ that are
orthogonal in $L^2_{[0,1]}$ to the functions
\begin{equation}\label{ad:1}
g_1^i ( 0\leq i\leq 2^2-1), r_{k_1}\quad \mbox{and}\quad
h_{1}^{(k_1+1)}.
\end{equation}
By (\ref{rad:1}) it is obvious that the functions (\ref{ad:1})
constitute an orthonormal set of functions. According to our
construction the set of the functions $$ \{{g}_{1}^i(x)\}
_{i=0}^{3}\bigcup \{r_{k_1}\}\bigcup \{h_{1}^{({k_1+1})}\}
\bigcup \{\psi_\nu(x)\} _{\nu=1}^{m_{1} }$$ is an orthonormal
basis in ${\mathcal E}^{k_1+2}_{[0,1]}.$

 At the $n-$th
step, $n>1,$ of our construction we
 define
\begin{equation*}\label{eq:6}
\widehat{g}_{n}^i(x)=\left\{\begin{array}{lll} n^{-1}2^{n^{2}}
f_{2n}^i\bigg(2^{2n^2+2}(x- 2^{-2n^2})\bigg)\quad &{\mbox {if
}}\quad x\in [0,2^{-2n^2}];\\ n^{-1} 2^{\frac{k}2}
f_{2n}^i\bigg(2^{k+2}(x- 2^{-k})\bigg) \quad &{\mbox {if }}\quad
x\in (2^{-k},2^{-k+1}],
\\ \quad &{\mbox {and }}\quad
 2\leq k\leq 2n^2;\\
 r_{k_{n-1}+4+i}(x)I_{(\frac12,1]}, \qquad &{\mbox {if }}\quad x\in
 (\frac12,1],
\end{array}\right.
\end{equation*}
 and $0\leq i\leq 2^{2n}-1.$ Then as above we extend
the functions $\widehat{g}_{n}^i, 0\leq i\leq 2^{2n}-1$
periodically with period $1$ to the whole line and denote
\begin{equation}\label{eq:7}
 {g}_{n}^i(x)=
\widehat{g}_{n}^i(2^{k_{n-1}+2n^2+2}x) \qquad {\mbox {for
all}}\quad 0\leq i\leq 2^{2n}-1.
\end{equation}
  It is easy to check that the functions
$\{{g}_{n}^i(x): 0\leq i\leq 2^{2n}-1\}$ are orthonormal in the
space $L^2_{[0,1]}$ and if
 $k_{n}$ is the smallest natural number
such that
$$
g_{n}^i \in {\mathcal E}^{k_{n}}_{[0,1]}  \qquad
{\mbox {for all}}\quad 0\leq i\leq 2^{2n}-1,
$$
then by the definition of the set of functions $\{\widehat{g}_{n}^i(x): 0\leq i\leq 2^{2n}-1\}$
 (\ref{eq:7}) and Lemma \ref{lem:1}
$$
g_{n}^i\in {\mathcal E}^{k_{n-1}+1,k_{n}}_{[0,1]} \quad 0\leq i\leq
2^{2n}-1.
$$
 We take a set of orthonormal functions
  \[
 \psi_\nu \in
{\mathcal E}^{k_{n-1}+2,k_{n}+2}_{[0,1]}, m_{n-1}+1\leq \nu \leq
2^{k_{n}+2 } - 2^{2n}-2 = m_{n},
\]
 that are orthogonal in
$L^2_{[0,1]}$ to the functions
\begin{equation*}\label{add:1}
r_{k_{n}}, h_{1}^{({k_{n}+1})}\quad \mbox{and}\quad  g_n^i \quad
\mbox{for all}\quad  (0\leq i\leq 2^{2n}-1).
\end{equation*}
As above we conclude  that the set of functions
$$
 \{{g}_{n}
^i(x)\}_{i=0}^{2^{2n}-1}\bigcup \{r_{{k_{n}}}(x)\} \bigcup
\{h_{1}^{({k_{n}+1})}(x)\} \bigcup \{\psi_\nu(x)\}
_{\nu=m_{n-1}+1}^{m_{n}}
 $$
  is an orthonormal basis in ${\mathcal
E}^{k_{n-1}+2,k_{n}+2}_{[0,1]}.$ Hence,
\begin{equation*}\label{ausys:1}
\bigcup_{n=1}^{\infty}\{{g}_{n}^i(x)\} _{i=0}^{2^{2n}-1} \bigcup
\{r_{{k_{n}}}(x)\}_{n=1}^{\infty}\bigcup
\{h_{1}^{({k_{n}+1})}(x)\}_{n=1}^{\infty}
 \bigcup \{\psi_\nu(x)\} _{\nu=1}^{\infty}
 \end{equation*}
  is a CONS in $L^2_{[0,1]}.$ From our construction if follows
  immediately that
\begin{equation*}\label{eq:8}
\int_{\Delta^{(k)}_{\nu}} \bigg( \sum_{i=
0}^{2^{2n}-1}{a}_{i}{g}_{n}^{i}(x) \bigg)^2 dx =
|\Delta^{(k)}_{\nu}| \sum_{i= 0}^{2^{2n}-1}{a}_{i}^2
\end{equation*}
for any  $\Delta^{(k)}_{\nu} = (\frac{\nu-1}{2^{k}},
\frac{\nu}{2^{k}}),$ where $ 1\leq \nu \leq 2^k, 1\leq k \leq
k_{n-1}+2n^2+2.$
 Moreover, by (\ref{eq:4})
we obtain that for any $\omega\in {\IR}$
\begin{equation}\label{eq:9}
\int_{\Delta^{(k)}_{\nu}} \bigg(\omega + \sum_{i=
0}^{2^{2n}-1}{a}_{i}{g}_{n}^{i}(x) \bigg)^2 dx =
|\Delta^{(k)}_{\nu}|\bigg( \omega^{2} +\sum_{i=
0}^{{2^{2n}-1}}{a}_{i}^2 \bigg).
\end{equation}
 We also have
that for any $n\in {\IN}$
\begin{equation}\label{ad:2}
 \| \psi_\nu\|_{\infty} \leq \sqrt{2^{k_n+2}} \quad
\mbox{for all}\quad m_{n-1}+1 \leq \nu \leq m_n,m_0=0
\end{equation}
just because they  belong to the space ${\mathcal
E}^{k_{n}+2}_{[0,1]}.$ Note also that
\begin{equation}\label{ad:4}
m_n < 2^{k_n+2}  \quad \mbox{for any}\quad n\in {\IN}.
\end{equation}

According to our construction and Lemma 3 of \cite{K:04} the
following assertion holds.
\begin{prop}\label{clm:1}
 For all $n\in \IN $ and for any collection
of nontrivial functions $$ F_j(x) = \sum_{i=1}^{2^{j+2}-1}
a_i^{(j)} {g}_{j}^i(x)\quad 1\leq j\leq n$$ the functions  $
\{F_j(x), r_{k_j}(x)\}_{j=1}^{n}$ constitute a set of independent functions.
\end{prop}
The following propositions were proved in \cite{K:06}
\begin{prop}\label{p:2}
For any sequence $\{c_{\nu}\} _{\nu=1}^{\infty}\in l^2$
\begin{equation*}\label{add:2}
\int_{[0,1]}\sup_{n}
|\sum_{\nu=1}^{m_n}c_{\nu}\psi_{\nu}(x)|^{2}dx \leq C
\sum_{\nu=1}^{\infty}c_{\nu}^2,
\end{equation*}
 for some $C>0$
independent of the coefficients.
\end{prop}

\begin{prop}\label{p:21}
The system $\{{g}_{n}^i(x): 0\leq i\leq 2^{2n}-1\}
_{n=1}^{\infty}$ is an orthonormal system of convergence.
  \end{prop}

 \

\subsection{ Construction of the second auxiliary CONS}
  In
this section our aim is to  transform  the set of orthogonal
functions $$ \{h_{1}^{({k_{n}+1})}(x)\}_{n=1}^{\infty}
 \bigcup \{\psi_\nu(x)\} _{\nu=1}^{\infty} $$  into
an orthonormal system of convergence $\{\xi_l(x)\}
_{l=1}^{\infty}.$ We will do that by the help of the orthogonal
matrices (see \cite{K:04}, Proposition 1)
\[
{\mathcal K}_N = (\kappa_{ij}^{(N+1)}) = \left(
\begin{array}{ccccc} { \frac1{\sqrt{N+1}}} & { \frac1{\sqrt{N+1}}}
& { \frac1{\sqrt{N+1}}} & \cdots & { \frac1{\sqrt{N+1}}}\\
1-\delta_N & -\delta_N & \cdots & -\delta_N  & {
\frac1{\sqrt{N+1}}} \\ -\delta_N & 1-\delta_N &
 \cdots & -\delta_N & {
\frac1{\sqrt{N+1}}}\\ \cdot & \cdot & \cdot &
 \cdots & \cdot\\
 -\delta_N & -\delta_N &
 \cdots & 1-\delta_N & {
\frac1{\sqrt{N+1}}}
\end{array} \right),
 \]
 where $\delta_N = {
\frac1{N}}(1+ { \frac1{\sqrt{N+1}}}). $
 Moreover, we will obtain some estimates
on $\|\xi_l\|_{L_{[\eta_l,1]}^\infty}$, where $\eta_l\to 0$ as
$l\to \infty.$
 Let
 $$
 q_0 =0, \quad q_n =
\left(2^{2(k_n+1)}-1\right)(m_n - m_{n-1}) \quad \mbox{for
any}\quad n\in {\IN}
$$
and put
\begin{equation}\label{add:3}
  p(n) =  2^{2(k_n+1)} \quad \mbox{and} \quad q_{\nu}(n) = \sum_{i=1}^{n-1}q_{i}+
(\nu-m_{n-1}-1) (2^{2(k_n+1)}-1)
\end{equation}  for all $n\in
{\IN}$ and $\nu ( m_{n-1}+1 \leq \nu \leq m_n).$

 Afterwards for any $\nu ( m_{n-1}+1
\leq \nu \leq m_n)$  and $1\leq j \leq p(n)$ we define
\begin{equation}\label{add:4}
 \phi_{j}^{\nu}(x) =
\kappa_{1j}^{(p(n))}\psi_{\nu}(x) +
\sum_{i=2}^{p(n)}\kappa_{ij}^{(p(n))}h_{1}^{k_{q_{\nu}(n)+i-1}+1}(x).
\end{equation}
By (\ref{ad:2})  and (\ref{add:3}),(\ref{add:4}) we have that for
any $ \nu\in [m_{n-1}+1, m_n]\cap{\IN}$
\begin{equation}\label{ad:3}
|\phi_{j}^{\nu}(x)| \leq \sqrt{2^{-k_n}}\quad \forall \, x\in
[2^{-k_{q_{\nu}(n)}-1},1]\quad\mbox{and}\quad \forall j\in
[1,p(n)]\cap{\IN}.
\end{equation}
From (\ref{ad:2}),(\ref{ad:4}) and (\ref{add:4}) follows that for
any $\nu\in [m_{n-1}+1, m_n]\cap{\IN}$ and any $n\in {\IN}$
\begin{equation}\label{phi:1}
\phi_{j}^{\nu}\cdot I_{[2^{-k_{n}-2},1]}\in {\mathcal
E}^{k_n+2}_{[0,1]} \quad \mbox{for }\quad 1\leq j\leq p(n).
\end{equation}
 In order to enumerate the obtained functions we
  put
 \begin{equation}\label{ro:1}
  \rho_0 =0, \, \rho_n = \sum_{j=1}^{n} 2^{2(k_j+1)}(m_j - m_{j-1})
\end{equation}
and denote $$ \rho_{\nu}(n) = \rho_{n-1}+ (\nu-m_{n-1}-1)
2^{2(k_n+1)} \, $$ for all $n\in {\IN}$ and $\nu ( m_{n-1}+1 \leq
\nu \leq m_n).$ Afterwards we denote
\begin{equation}\label{xi:1}
\xi_{l}(x) =
\phi_{j}^{\nu}(x)\quad \mbox{if}\quad l= \rho_{\nu}(n)+j, 1\leq
j\leq p(n)
\end{equation}
 and $m_{n-1}+1 \leq \nu \leq m_n.$ Hence, from
(\ref{phi:1}) it follows that for any $l\in {\IN}$ such that
\begin{equation}\label{el:1}
 \rho_{n-1}< l \leq \rho_{n}
 \end{equation}
 and any dyadic interval
$\Delta^{(k_n+2)}_{\nu} = (\frac{\nu-1}{2^{k_n+2}},
\frac{\nu}{2^{k_n+2}}),$ where $2\leq \nu \leq 2^{k_n+2}$ the
function $\xi_l$ is constant on the interval
$\Delta^{(k_n+2)}_{\nu}:$
\begin{equation}\label{cl:1}
 \xi_l(x) = \omega_{\nu}^{l} \quad
\mbox{for }\quad x\in \Delta^{(k_n+2)}_{\nu}\quad (2\leq \nu \leq
2^{k_n+2}).
\end{equation}
 Thus,
according to our construction, the set of functions $$
\bigcup_{n=1}^{\infty}\{{g}_{n}^i(x)\} _{i=0}^{2^{2n}-1}\bigcup
\{r_{{k_{n}}}(x)\}_{n=1}^{\infty} \bigcup \{\xi_l(x)\}
_{l=1}^{\infty} $$ is a CONS in $L^2_{[0,1]}.$

 We also have
\begin{prop}\label{p:3}
For any sequence $\{c_{l}\} _{l=1}^{\infty}\in l^2$
\begin{equation}
\int_{[0,1]}\sup_{n} |\sum_{l=1}^{m_n}c_{l}\xi_l(x)|^{2}dx \leq C
\sum_{l=1}^{\infty}c_{l}^2,
\end{equation}
 for some $C>0$
independent of coefficients.
  \end{prop}
  \begin{proof} The Proposition  \ref{p:3} follows immediately from
  Proposition \ref{p:2} and Lemma \ref{le:11}. We should check that the conditions of Lemma \ref{le:11} are
  satisfied for the system  $\{h_{1}^{({k_{n}+1})}(x)\}_{n=1}^{\infty}$.

  If we denote $h(x)= \sum_{n=1}^{\infty}a_{k_{n}}h_{1}^{({k_{n}+1})}(x)$ then
its orthogonal projection  $P(f)$ onto the subspace ${\mathcal
E}^{k_m+2}_{[0,1]}$   equals
\[
|\Delta^{(k_m+2)}_{\nu}|^{-1}\int_{\Delta^{(k_m+2)}_{\nu}} h(t)dt\,
\mbox{ on the interval}\, \Delta^{(k_m+2)}_{\nu}, 1\leq \nu\leq
2^{k_m+2}.
\]
 Thus $|\sum_{n=1}^{m}a_{k_{n}}h_{1}^{({k_{n}+1})}(x)|$ can be
estimated  by the Hardy-Littlewood maximal function of $f$. Hence,
applying the  $L^2_{[0,1]}\rightarrow L^2_{[0,1]} $ boundedness of
the indicated operator (cf. \cite{Zy:59}) we finish the proof.
\end{proof}

\

\subsection{ Construction of a Euclidean system of bounded functions   }

On this step we construct a Euclidean system of bounded functions by
 transformation of finite collections of functions from the ONS
  \[
\bigcup_{n=1}^{\infty}\{{g}_{n}^i(x)\} _{i=0}^{2^{2n}-1} \bigcup
\{\xi_l(x)\} _{l=1}^{\infty}
\]
  applying  the orthogonal matrices
${\mathcal K}_N.$ For any $n\in {\IN}$ we put $l(n) = 2^{2n}+1$ and define
 \begin{equation}\label{up:01}
 \Upsilon_{j}^{n}(x) =
\kappa_{1j}^{(l(n))}\xi_{n}(x) +
\sum_{i=2}^{2^{2n}+1}\kappa_{ij}^{(l(n))}{g}_{n}^{i-2}(x) \quad
1\leq j \leq 2^{2n}+1.
\end{equation}
  Evidently, the obtained system
of functions
 \[
 \left\{ \{{\Upsilon}_{j}^n(x)\} _{j=1}^{2^{2n}+1}
\right \}_{n=1}^{\infty}
\]
 is again an ONS. We enumerate them in
the natural order: for
 \begin{equation}\label{up:02}
 \mu_{0} = 0, \quad  \mu_{m} = \quad
\mu_{m-1}+ 2^{2m}+ 1
\end{equation}
we put
\begin{equation}\label{up:1}
 \Upsilon_{k}(x) =
\Upsilon_{j}^{m}(x) \quad \mbox{if}\quad k=  \mu_{m-1} + j, \quad
\mbox{where}\quad 1\leq j \leq 2^{2m}+1.
\end{equation}
\begin{thm}\label{t:10}
The ONS $\{\Upsilon_{k}(x)\}_{k=1}^{\infty}$ is
a system of convergence.
\end{thm}
\begin{proof} The theorem is an immediate consequence of Lemma
\ref{le:11} and Propositions \ref{p:21} and \ref{p:3}.
\end{proof}

\begin{thm}\label{t:11} The ONS $\{\Upsilon_{k}(x)\}_{k=1}^{\infty}$ is
a system of divergence.
\end{thm}
The proof of Theorem \ref{t:11}  is a particular case of the proof
of Theorem \ref{t:12} which will be given  in the next section.
However, while explaining the idea of the construction of the
system $\Theta$  we will refer the system
$\{\Upsilon_{k}(x)\}_{k=1}^{\infty}$ as a Euclidean system.

\

\section{A uniformly bounded  complete Euclidean system }\label{ubos:1}

We have constructed a Euclidean system
$\{\Upsilon_{k}(x)\}_{k=1}^{\infty}$ such that
\begin{equation} \label{sis01}
\{\Upsilon_{k}(x)\}_{k=1}^{\infty}\bigcup
\{r_{k_n}(x)\}_{n=1}^{\infty}
\end{equation}
is a complete ONS. If
we enumerate the system (\ref{sis01}) in some order then {\rm a
priori} it is not clear that the obtained system will be a complete Euclidean system .
 Evidently it will be a convergence system.  Whether
the obtained system is a divergence system or not is far from being
clear. In our particular case this problem  is solved mainly with
the help of Proposition \ref{clm:1}. Moreover, for the proof of Theorem
\ref{t:1} we need a stronger property which we explain after
enumerating the system $\{\Upsilon_{k}(x)\}_{k=1}^{\infty}\bigcup
\{r_{k_n}(x)\}_{n=1}^{\infty}$ in a special  way. Let $M> \sqrt2 +1,$ be the constant from Theorem \ref{t:1} and let $l_{0} \in \mathbb{N}$ be such that
$$\sqrt{2^{-l_{0}}} < M -\sqrt2 -1.$$

 Afterwards,  we put
\begin{equation*}\label{om:0}
 \chi_{0}^{(i)}(x) =   r_{k_{i}}(x), \quad  1 \leq i\leq \nu_0,
\end{equation*}
where $\nu_0\in \mathbb{N} $ is such that
\begin{equation*}\label{om:01}
 \Upsilon_{1} \in  {\mathcal E}^{k_{\nu_0}}_{[0,1]}.
\end{equation*}

Then we define
\begin{equation*}\label{om:1}
 \chi_{1}^{(1)}(x) = \Upsilon_{1}(x)\qquad \mbox{and}\quad
 \chi_{1}^{(i)}(x)= r_{k_{\nu_0+ i-1}}(x), \quad  1 < i \leq
 2^{n_1},
\end{equation*}
where $n_1\in \IN$ is such that $ n_{1} \geq  k_{\nu_0}+ l_{0}$
\begin{equation*}\label{om:11}
 \Upsilon_{2} \in  {\mathcal E}^{k_{\nu_1}}_{[0,1]}\quad \mbox{and}\quad \nu_{1} = 2^{n_1}+\nu_0.
\end{equation*}

In the same way for any $j\geq 2 $ we define
\begin{equation}\label{om:2}
 \chi_{j}^{(1)}(x) = \Upsilon_{j}(x)\qquad \mbox{and}\quad
 \chi_{j}^{(i)}(x)= r_{k_{\nu_{j-1}+ i-1}}(x), \quad  1 < i \leq
 2^{n_j},
\end{equation}
where  $n_j\in
\IN$ is such that $n_{j} \geq  k_{\nu_{j-1}} + l_{0}$
\begin{equation}\label{om:21}
 \Upsilon_{j+1} \in  {\mathcal E}^{k_{\nu_{j}}}_{[0,1]} \quad   \mbox{and}\quad \nu_{j}= 2^{n_{j}}+\nu_{j-1}.
\end{equation}

If we transform the functions
\begin{equation}\label{om:22}
 \Upsilon_{j}(x),  r_{k_{\nu_{j-1}}+i}(x), \quad  1 < i \leq
 2^{n_j}
\end{equation}
by the    Haar matrix $H_{n_j}$  then it is easy to check that the obtained  orthonormal functions are
bounded by the constant $M.$

 By (\ref{up:01})--(\ref{up:1}) we can consider that the numbers $k_{\nu_{j-1}}$ are chosen so that
 for any $m\in \IN$
\begin{equation}\label{om:23}
 \Upsilon_{j } \in  {\mathcal E}^{k_{\nu_{\mu_{m-1}}}}_{[0,1]}
  \quad \mbox{for all}\quad \mu_{m-1}+1\leq j\leq
 \mu_m.
\end{equation}

 Denote
\begin{equation}\label{om:3}
 \chi_{k}(x)=
 \chi_{0}^{(k)}(x), \qquad \mbox{where}\quad
 1 \leq k \leq \nu_0
\end{equation}
\begin{equation}\label{om:31}
 \chi_{k}(x)=
 \chi_{j}^{(i)}(x), \qquad \mbox{where}\quad  k= \nu_{j-1} + j +i-1;
 \end{equation}
\[ 1 \leq i \leq \nu_{j-1} +  2^{n_j} \quad \mbox{and}\quad j\geq
1.\]

Then the system $\{\chi_{k}(x)\}_{k=1}^{\infty}$ will be a
complete ONS.
 Moreover, the following
assertion is true.
\begin{thm}\label{t:12} For any $\{a_{k}\}_{k=1}^{\infty}\notin l^2$
the partial sums
\begin{equation}\label{ser:0}
S_j(x) = \sum_{k=1}^{\nu_{j} + j-1} a_k\chi_{k}(x), \qquad j\in \IN,
\end{equation}
 diverge a.e. on $[0,1],$ when $j\to +\infty.$
\end{thm}

\subsection{  Proof of  Theorem \ref{t:12}}

 Let us rewrite the
series $\sum_{k=1}^{\infty} a_k\chi_{k}(x)$ as
\begin{equation}\label{ser:1}
\sum_{i=1}^{\nu_{0}} a_0^{(i)}\chi_{0}^{(i)}(x) +
\sum_{j=1}^{\infty} \sum_{i=1}^{2^{n_{j}}}
a_j^{(i)}\chi_{j}^{(i)}(x)
\end{equation}
and observe that the partial sums $$ \sum_{i=1}^{\nu_{0}}
a_0^{(i)}\chi_{0}^{(i)}(x) + \sum_{j=1}^{N} \sum_{i=1}^{2^{n_{j}}}
a_j^{(i)}\chi_{j}^{(i)}(x) $$ coincide with the corresponding
partial sums (\ref{ser:0}).

 If $\sum_{j=1}^{\infty}
|a_j^{(1)}|^2 < +\infty$ then by Theorem \ref{t:10} and well known
properties of the Rademacher system (see \cite{Zy:59}) we
immediately obtain that the sequence (\ref{ser:0}) diverges a.e.
when $j\to +\infty.$ Thus we have to consider only the case when
\begin{equation}\label{ser:2}
\sum_{k=1}^{\infty} |a_k^{(1)}|^2 = +\infty.
\end{equation}
For any $m\in \IN$ we put
 \begin{equation}\label{ser:3}
\alpha_{j}^{m}= a_{k}^{(1)}  \quad \mbox{if} \quad k= \mu_{m-1}
+j, \quad \mbox{where} \quad 1\leq j \leq 2^{2m}+1.
\end{equation}
 According to
the construction we have
\begin{eqnarray}\label{ser:4}
\sum_{k=\mu_{m-1}+1}^{\mu_{m}} \sum_{i=1}^{2^{n_{k}}}
a_k^{(i)}\chi_{k}^{(i)}(x) &=&
\sum_{k=\mu_{m-1}+1}^{\mu_{m}}a_k^{(1)} \Upsilon_{k}(x) +
\Psi_m(x)\\ &=& \sum_{j=
1}^{2^{2m}+1}\alpha_{j}^{m}\Upsilon_{j}^{m}(x) + \Psi_m(x)\\ &=&
\beta_{m}\xi_{m}(x) + \sum_{i=
0}^{2^{2m}-1}\gamma_{m}^{i}g_{m}^{i}(x) + \Psi_m(x),
\end{eqnarray}
where
 \begin{equation}\label{ser:5}
\Psi_m(x) = \sum_{k=\mu_{m-1}+1}^{\mu_{m}} \sum_{i=2}^{2^{n_{k}}}
a_k^{(i)}\chi_{k}^{(i)}(x) = \sum_{k=\mu_{m-1}+1}^{\mu_{m}}
\sum_{i=2}^{2^{n_{k}}} a_k^{(i)} r_{k_{i+\nu_{j-1}-1}}(x)
\end{equation}
 and
\[
 \beta_{m} = \frac1{\sqrt{2^{2m}+1}} \sum_{j=
1}^{2^{2m}+1}\alpha_{j}^{m} = \int_{0}^{1}[\sum_{j=
1}^{2^{2m}+1}\alpha_{j}^{m}\Upsilon_{j}^{m}(x)]\xi_{m}(x) dx.
\]
Let
 \[
  {\mathcal
M}_{m}^{2} = \sum_{k=\mu_{m-1}+1}^{\mu_{m}}|a^{(1)}_{k}|^2 =
(\beta_{m})^2 + \sum_{i= 0}^{2^{2m}-1}(\gamma_{m}^{i})^2.
\]

The Cauchy inequality yields
 $$ | \beta_{m}| \leq {\mathcal
M}_{m} \quad \mbox{for all} \quad m\in{\IN}.$$

 For any  $ \epsilon \in (0,1] $ we denote $$ \Omega_{\epsilon} =
\left\{m\in{\IN}: \frac{\epsilon}4{\mathcal M}_{m} \leq  |
\beta_{m}|\right\}$$ and  $$ \Omega_{\epsilon}^c = {\bf
N}\setminus \Omega_{\epsilon} = \left\{m\in{\IN}: \frac{\epsilon}4
{\mathcal M}_{m}
>  | \beta_{m}|\right\}.$$ Evidently $ \Omega_{\epsilon} \supseteq
 \Omega_{\delta}$ for any $0 < \epsilon < \delta \leq 1.$
 The proof will be divided into two main
parts.

\

\subsubsection { The case  $\sum_{s =1 }^{\infty} {\mathcal
M}_{m(s)}^{2} = +\infty,$ { for some} $\{ m(s) \}_{s=1}^{\infty}
\subset \Omega_{\epsilon}$}

 Suppose $\{ m(s) \}_{s=1}^{\infty} \subset \Omega_{\epsilon}$ is any subsequence of
 natural numbers in $\Omega_{\epsilon}$ such that
 \begin{equation}\label{eq:10}
 \sum_{s =1 }^{\infty} {\mathcal
M}_{m(s)}^{2} = +\infty. 
\end{equation}
 Without loss in
generality we can suppose that
\begin{equation}\label{ms:1}
 2^{-s} \leq {\mathcal
M}_{m(s)}
\end{equation}
 because (\ref{eq:10})
remains true after deleting the terms that does not satisfy   the
condition (\ref{ms:1}). Let
\begin{equation}\label{ser:11}
S_m^{v}(x) = \sum_{k=\mu_{m-1}+1}^{v} \sum_{i=1}^{2^{n_{k}}}
a_k^{(i)}\chi_{k}^{(i)}(x)= \sum_{k=\mu_{m-1}+1}^{v}a_k^{(1)}
\Upsilon_{k}(x) + \Psi_m^{v}(x),
\end{equation}
where $ \mu_{m-1}+1 \leq v \leq \mu_{m}  $ and 
 \begin{equation}\label{ser:12}
\Psi_m^{v}(x) = \sum_{k=\mu_{m-1}+1}^{v} \sum_{i=2}^{2^{n_{k}}}
a_k^{(i)}\chi_{k}^{(i)}(x) = \sum_{k=\mu_{m-1}+1}^{v}
\sum_{i=2}^{2^{n_{k}}} a_k^{(i)} r_{k_{i+\nu_{k-1}-1}}(x).
\end{equation}

 We put
 $$
  \sigma_{m}^*(x) = \sup_{1\leq l
\leq 2^{2m}}\bigg| \sum_{j=1}^{l}
\alpha_{j}^{m}\Upsilon_{j}^{m}(x)\bigg|
$$
 $$ = \sup_{1\leq l \leq
2^{2m}}\bigg| \sum_{j=1}^{l} \alpha_{j}^{m} \bigg( \frac1{\sqrt{
2^{2m}+1}} \xi_{m} +
\sum_{i=2}^{2^{2m}+1}\kappa_{ij}^{(2^{2m}+1)}{g}_{m}^{i-2}(x)\bigg)\bigg|.
$$
 By (\ref{cl:1}),(\ref{el:1}) we will have   that   for
\begin{equation}\label{ro:2}
 \rho_{n-1}< m \leq \rho_n \quad \mbox{and}\quad x\in\Delta^{(k_n+2)}_{\nu}, 2\leq \nu \leq
2^{k_n+2}
 \end{equation}
   $$ \sigma_{m}^*(x) =
 \sup_{1\leq l
\leq 2^{2m}} \bigg| \sum_{j=1}^{l} \alpha_{j}^{m} \bigg(
\frac1{\sqrt{2^{2m}+1}} \omega_{\nu}^{m} +
\sum_{i=2}^{2^{2m}+1}\kappa_{ij}^{(2^{2m}+1)}{g}_{m}^{i-2}(x)\bigg)
\bigg|.  $$ 
If we denote by $n_s$ the number that corresponds to $n$   in
the condition (\ref{ro:2}) when $m$ is replaced by $m(s)$ it is easy to observe  (see (\ref{ro:1}),
(\ref{ro:2})) that
\begin{equation*}\label{ro:3}
 m(s) > n_s+2 \quad \mbox{for all}\quad s\geq 4.
 \end{equation*}
 Thus if $x\in\Delta^{(k_{m-1}+2)}_{\nu},$ where $\nu$  is such that
  $ 2^{k_{m-1}-k_{n+1}}\leq \nu \leq 2^{k_{m-1}+2}$
then
$$ \sigma_{m}^*(x) =
 \sup_{1\leq v
\leq 2^{2m}} \bigg| \sum_{j=1}^{v} \alpha_{j}^{m} \bigg(
\frac1{\sqrt{2^{2m}+1}} \omega_{\nu}^{*} +
\sum_{i=2}^{2^{2m}+1}\kappa_{ij}^{(2^{2m}+1)}{g}_{m}^{i-2}(x)\bigg)
\bigg|,
 $$
  where $\omega_{\nu}^{*} = \omega_{\nu_1}^{m}$ for $
\Delta^{(k_{m-1}+2)}_{\nu} \subset \Delta^{(k_n+2)}_{\nu_1}.$

For any $s\in \IN$ we define $\tau_s\in \IN$ so that
\begin{equation}\label{ka:1}
\sqrt{ 2^{\tau_s -1}} {\mathcal M}_{m(s)} < 1 \leq   \sqrt{
2^{\tau_s }} {\mathcal M}_{m(s)},
 \end{equation}
otherwise we put $\tau_s =2.$

Hence by (\ref{ms:1}) we will have that
\begin{equation*}\label{tau:2}
2\leq \tau_s \leq 2s+1 \qquad \mbox{for all }\quad s\in \IN.
 \end{equation*}
Define
 \begin{equation*}\label{ke:1}
E_{s,\nu}= \{x\in\Delta^{(k_{m(s)-1}+2)}_{\nu}:
I_{[2^{-\tau_s},2^{-\tau_s+1}]}(2^{k_{m(s)-1}+2{m(s)}^2+2}x) =1\},
 \end{equation*}
 \[
1 \leq \nu \leq 2^{k_{m(s)-1}+2},
 \]
 where we suppose that the characteristic function $I_E$ is
 extended   with the period $1$ to the whole line.
 Evidently,
 \begin{equation}\label{em:1}
|E_{s,\nu}|= 2^{-\tau_s} |\Delta^{(k_{m(s)-1}+2)}_{\nu}|.
 \end{equation}
Then we write
 \begin{eqnarray*}
  \sigma_{m(s)}^*(x) &\geq& \sup_{1\leq v  \leq  2^{2m(s)}}|
\sum_{j=1}^{v} \alpha_{j}^{m(s)} {g}_{m(s)}^{j-1}(x)|\\ &-&
\delta_{2^{2m(s)}+1}\bigg| \frac{\delta_{2^{2m(s)}+1}^{-1}
\omega_{\nu}^{*} }{\sqrt{ 2^{2m(s)}+1}} - \sum_{j=1}^{2^{2m(s)}}
 {g}_{m(s)}^{j-1}(x)\bigg| \sum_{j=
1}^{2^{2m(s)}}|\alpha_{j}^{m(s)}| \\ &=&{\widehat
\sigma}_{m(s)}^*(x) - {\widehat  R}_{\nu}^{s}(x)\quad
\mbox{if}\quad x\in E_{s,\nu}
\end{eqnarray*}
$$(2^{k_{m(s)-1}-k_{n_s+1}}\leq
\nu \leq 2^{k_{m(s)-1}+2}).$$
 Applying
the equality (\ref{eq:9}) we obtain
\begin{eqnarray*}
\lefteqn{\int_{E_{s,\nu}} {\widehat R}_{\nu}^{s}(x) dx \leq
\sqrt{|E_{s,\nu}|} \left( \int_{E_{s,\nu}} ({\widehat
R}_{\nu}^{s}(x))^2 dx \right)^{\frac{1}{2}}}
 \\ &\leq& 4 {\mathcal M}_{m(s)}\sqrt{|E_{s,\nu}|}\frac{1}
{\sqrt{2^{2m(s)}+1}}\\
&\times&
 \left( \int_{E_{s,\nu}} \bigg( (2^{2m(s)}+1)
(\omega_{\nu}^{*})^{2} + \sum_{j=
1}^{2^{2m(s)}+1}({g}_{m(s)}^{j-1}(x))^2 \bigg) dx
\right)^{\frac{1}{2}}
\\ &\leq& 4|E_{s,\nu}|{\mathcal M}_{m(s)}
\sqrt{(\omega_{\nu}^{*})^{2}+m(s)^{-2}} \\ &\leq& 4
\sqrt{ 2^{-k_{n_s}} +m(s)^{-2}}\, |E_{s,\nu}|\,{\mathcal M}_{m(s)}
\end{eqnarray*}
 for all $ \nu (2^{k_{m(s)-1}-k_{n_s+1}}\leq
\nu \leq 2^{k_{m(s)-1}+2}).$

 The last inequality follows from the conditions
 (\ref{ad:3}),(\ref{xi:1}) and (\ref{cl:1}).
On the other hand by the definition of the functions $\widehat{g}_{n}^i(x)$  and (\ref{eq:7}) we deduce
that for all
$
\nu (2^{k_{m(s)-1}-k_{n_s+1}}\leq \nu \leq
2^{k_{m(s)-1}+2})
$
\begin{eqnarray*}
 &\,&\int_{E_{s,\nu}} {\widehat \sigma}_{{m(s)}}^*(x)
dx\\
 &=& \int_{E_{s,\nu}} \sup_{1\leq v \leq 2^{{2m(s)}}}|
\sum_{j=1}^{v} \alpha_{j}^{{m(s)}}
\widehat{g}_{{m(s)}}^{j-1}(2^{k_{m(s)-1}+ 2m(s)^2 +2}x)| dx\\ &=&
|\Delta^{(k_{m(s)-1}+2)}_{\nu}|
 \int_{2^{-\tau_s}}^{2^{-\tau_s+1}} \sup_{1\leq v \leq
2^{{2m(s)}}}| \sum_{j=1}^{v} \alpha_{j}^{{m(s)}}
\widehat{g}_{{m(s)}}^{j-1}(x)| dx\\
 &=&
m(s)^{-1}2^{\frac{\tau_s}2}|\Delta^{(k_{m(s)-1}+2)}_{\nu}|\\
&\times & \int_{2^{-\tau_s}}^{2^{-\tau_s+1}} \sup_{1\leq v \leq 2^{2m(s)} }|
\sum_{j=1}^{v} \alpha_{j}^{{m(s)}} f_{2m(s)}^{j-1}(2^{\tau_s+2}x-
4)| dx\\
  &=&
m(s)^{-1}2^{-\frac{\tau_s}2}|\Delta^{(k_{m(s)-1}+2)}_{\nu}|
 \int_{0}^{4} \sup_{1\leq v \leq 2^{2m(s)} }| \sum_{j=1}^{v}
\alpha_{j}^{{m(s)}} f_{2m(s)}^{j-1}(y))| dy \\
 &\geq&
2^{\frac{\tau_s}2}\frac{|E_{s,\nu}|}{6} 2^{-m(s)}
 \bigg| \sum_{j=1}^{2^{2m(s)}} \alpha_{j}^{{m(s)}}\bigg|.
 \end{eqnarray*}
  In
order to obtain the last two inequalities we have applied
consecutively Lemma 2 and Lemma 1 of \cite{K:04}. In the last
inequality we applied also the equality (\ref{em:1}).
 We have that
  \begin{eqnarray*}
  |\sum_{j=1}^{2^{2m(s)}} \alpha_{j}^{m(s)}| &\geq& |
\sum_{j=1}^{2^{{2m(s)}}+1} \alpha_{j}^{{m(s)}}| -|
\alpha_{2^{{2m(s)}}+1}^{{m(s)}}|\\ & \geq & \bigg|
\sum_{j=1}^{2^{{2m(s)}}+1} \alpha_{j}^{{m(s)}}\bigg| - {\mathcal
M}_{{m(s)}} \geq \frac{1}{2} \bigg| \sum_{j=1}^{2^{{2m(s)}}+1}
\alpha_{j}^{{m(s)}}\bigg|
\end{eqnarray*}
 for any $ m(s)\geq
N_\epsilon,$ where  $N_\epsilon = 2\log_2 (\frac{8}{\epsilon}).$
 Hence, for any  $ m(s)\geq N_\epsilon,$ we have $$
\int_{E_{s,\nu}} {\widehat \sigma}_{{m(s)}}^*(x) dx \geq
2^{\frac{\tau_s}2}\frac{|E_{s,\nu}|}{12} 2^{-m(s)}
 \bigg| \sum_{j=1}^{2^{2m(s)}+1} \alpha_{j}^{{m(s)}}\bigg|.$$
  Thus if we take ${\widehat N}_\epsilon \geq  N_\epsilon$ so that
$$ 4 \sqrt{2^{-k_{n_s}} +m(s)^{-2}} <
\frac{\epsilon}{96}\qquad\mbox{when}\quad m(s)>{\widehat
N}_\epsilon $$
 then   we will obtain that for all $\nu (2^{k_{m(s)-1}-k_{n_s+1}}\leq \nu \leq 2^{k_{m(s)-1}+2})$
\begin{eqnarray*}
 \int_{E_{s,\nu}} { \sigma}_{{m(s)}}^*(x)
dx & \geq&
 \int_{E_{s,\nu}} {\widehat \sigma}_{{m(s)}}^*(x)
dx  - \int_{E_{s,\nu}} {\widehat R}_{\nu}^{{m(s)}}(x) dx\\
  &\geq& 2^{\frac{\tau_s}2}\frac{|E_{s,\nu}|}{6} 2^{-m(s)}
 \bigg| \sum_{j=1}^{2^{2m(s)}+1} \alpha_{j}^{{m(s)}}\bigg|\\
 &-& 4 \sqrt{2^{-k_{n_s}}+m(s)^{-2}} |E_{s,\nu}|{\mathcal
M}_{{m(s)}}\\ &\geq&
2^{\frac{\tau_s}2}\frac{|E_{s,\nu}|}{96}\epsilon{\mathcal
M}_{{m_s}}.
\end{eqnarray*}
 
Applying the Menshov-Rademacher
theorem
 by  the definition of the functions $\widehat{g}_{n}^i(x)$   and  (\ref{eq:7}) we have that
\begin{eqnarray*}
 &\,&\int_{E_{s,\nu}} |{\widehat
\sigma}_{{m(s)}}^*(x)|^2 dx\\ &=& \int_{E_{s,\nu}} \sup_{1\leq v \leq
2^{2m(s)}}| \sum_{j=1}^{v} \alpha_{j}^{m(s)}
\widehat{g}_{m(s)}^{j-1}(2^{k_{m(s)-1}+ 2m(s)^2 +2}x)|^2 dx\\  &=&
|\Delta^{(k_{m(s)-1}+2)}_{\nu}|
 \int_{2^{-\tau_s}}^{2^{-\tau_s+1}} \sup_{1\leq v \leq
2^{{2m(s)}}+1}| \sum_{j=1}^{v} \alpha_{j}^{{m(s)}}
\widehat{g}_{{m(s)}}^{j-1}(x)|^2 dx\\
 &=&
m(s)^{-2}2^{{\tau_s}}|\Delta^{(k_{m(s)-1}+2)}_{\nu}|\\
&\times&\int_{2^{-\tau_s}}^{2^{-\tau_s+1}} \sup_{1\leq v \leq 2^{2m(s)} }|
\sum_{j=1}^{v} \alpha_{j}^{{m(s)}} f_{2m(s)}^{j-1}(2^{\tau_s+2}x-
4)|^2 dx
\\
&\leq&
 C^2 |\Delta^{(k_{m(s)-1}+2)}_{\nu}| \left(\frac{m(s)+2}{m(s)}\right)^2 {\mathcal
M}_{m(s)}^2 \leq 4C^2|\Delta^{(k_{m(s)-1}+2)}_{\nu}|{\mathcal
M}_{{m(s)}}^2,
\end{eqnarray*}
 where $C>0$ is
an absolute constant.
 Hence, if $m(s)>{\widehat N}_\epsilon$
\begin{eqnarray*} \left(
\int_{E_{s,\nu}} ({\sigma}_{{m(s)}}^* (x))^2 dx
\right)^{\frac{1}{2}} &\leq& \left(\int_{E_{s,\nu}} ({\widehat
\sigma}_{{m(s)}}^*(x))^2 dx \right)^{\frac{1}{2}}\\ +
\left(\int_{E_{s,\nu}} ({\widehat R}_{\nu}^{{s}}(x))^2(x) dx
\right)^{\frac{1}{2}} &\leq& 2C
|\Delta^{(k_{m(s)-1}+2)}_{\nu}|^{\frac{1}{2}} {\mathcal
M}_{{m(s)}}\\ +
|\Delta^{(k_{m(s)-1}+2)}_{\nu}|^{\frac{1}{2}}{\mathcal M}_{{m(s)}}
&\leq&  (2C +1) |\Delta^{(k_{m(s)-1}+2)}_{\nu}|^{\frac{1}{2}}
{\mathcal M}_{{m(s)}}
\end{eqnarray*} for any
$\nu (2^{k_{m(s)-1}-k_{n_s+1}}\leq \nu \leq 2^{k_{m(s)-1}+2}).$

 For the function  ${\sigma}_{{m(s)}}^*$ we
apply Lemma C of \cite{K:04} (see also \cite{K:93}, p.8)
 to estimate the
Lebesgue measure of the set $$ E_{s,\nu}^{*} = \bigg\{ { x}\in
E_{s,\nu}: {\sigma}_{{m(s)}}^*(x) \geq
{\frac{2^{\frac{\tau_s}2}\epsilon}{200}}{\mathcal M}_{{m(s)}}
\bigg\}.$$
 After normalizing the Lebesgue measure on the
set $E_{s,\nu}$ and easily computing the increase of the norms $L^1$ and
$L^2$  we will obtain that for any $\nu
(2^{k_{m(s)-1}-k_{n_s+1}}\leq \nu \leq 2^{k_{m(s)-1}+2}),$ ${m(s)}\in
\Omega_{\epsilon}, {m_s} \geq {\widehat N}_\epsilon$ 
$$
|E_{s,\nu}|^{-1} |E_{s,\nu}^{*}|  \geq \bigg(
\frac{2^{\frac{\tau_s}2}|E_{s,\nu}|^{\frac{1}{2}}\epsilon}{400(2C+1)
|\Delta^{(k_{m(s)-1}+2)}_{\nu}|^{\frac{1}{2}}}\bigg)^2$$ 
$$ =
\bigg( \frac{\epsilon}{400(2C+1)}\bigg)^2 = C_\epsilon>0. $$
By
(\ref{om:22}) and (\ref{om:23}) we observe (see
(\ref{ser:11}),(\ref{ser:12})) that for any $t\in \IR$
 $$ \bigg| \bigg\{ { x}\in
E_{s,\nu}:  S_m^{v}(x)
> t \bigg\}\bigg| = \bigg\{ { x}\in E_{s,\nu}:  S_m^{v}(x) - 2\Psi_m^{v}(x)
> t \bigg\}\bigg|.
$$ Hence if we define
\[
S_{m}^*(x) = \sup_{\mu_{m-1}+1\leq v \leq \mu_m} |S_m^{v}(x)|
\]
then we obtain that the measure of the set $$ \hat{E}_{s,\nu} =
\bigg\{ { x}\in E_{s,\nu}: S_{{m(s)}}^*(x) \geq
{\frac{2^{\frac{\tau_s}2}\epsilon}{200}}{\mathcal M}_{{m(s)}}
\bigg\}$$ is greater than or equal $\frac12 |E_{s,\nu}^{*}|.$

The
sequence of partial sums (\ref{ser:0}) diverges a.e. on the set
 $$E= \limsup_{s} \bigcup_{\nu
=2^{k_{m(s)-1}-k_{n_s+1}}}^{2^{k_{m_s-1}+1}} \hat{E}_{s,\nu} $$
 when $j\to +\infty.$
 By (\ref{em:1}) and
(\ref{ka:1}) it follows that for any dyadic interval
$$
\Delta^{(k_{m(s)-1}+2)}_{\nu}\,(2^{k_{m(s)-1}-k_{n_s+1}}\leq \nu
\leq 2^{k_{m(s)-1}+2})
$$
\begin{equation*}\label{div:1} \bigg|\bigg\{ { x}\in
\Delta^{(k_{m(s)-1}+2)}_{\nu}: S_{m}^*(x) \geq
{\frac{\epsilon}{200}} \bigg\}\bigg| \geq C_\epsilon
2^{-\tau_s-1}|\Delta^{(k_{m(s)-1}+2)}_{\nu}|
\end{equation*}
$$ \geq {\frac{1}{4}}C_\epsilon {\mathcal
M}^2_{{m_s}}|\Delta^{(k_{m(s)-1}+2)}_{\nu}|. $$ Hence one easily
derives that  $|E|=1.$

\

\subsubsection{ The case  $\sum_{m\in \Omega_{\epsilon}} {\mathcal
M}_{m}^{2} < +\infty$ { for all} $ \epsilon\in (0,1]$} The proof
of this part is similar to the proof given in \cite{K:04}.  In
this case we have that
\begin{equation}\label{eq:13} \sum_{m\in
\Omega^c_{\epsilon}} {\mathcal M}_{m}^{2} = +\infty \quad \forall
\epsilon\in (0,1].
\end{equation}
 Let $\{
m_{j}(\epsilon)\}_{j=1}^{\infty} = \Omega^c_{\epsilon},$ where
\[
m_1(\epsilon) < m_2(\epsilon) < \cdots < m_j(\epsilon) <
m_{j+1}(\epsilon) < \cdots
\]
 and
 study two subcases.

a) {\it When  $\limsup_{j} {\mathcal M}_{m_j(\epsilon)} >0$ for
some $\epsilon \in (0,1]);$}

b)  {\it When $\lim_{j\to \infty } {\mathcal M}_{m_j(\epsilon)}=0$
for any $\epsilon \in (0,1].$ }

In the case a) the reader must take into account the conditions
(\ref{om:2}) and (\ref{om:23}) to assure that the Rademacher
functions that appear between the functions $\Upsilon_{j }
(\mu_{m-1}+1\leq j\leq  \mu_m)$ do not affect on the proof.

The proof of Theorem \ref{t:12} in the case b) is also similar to
the proof given in \cite{K:04} for the corresponding case. Here
one should use the  conditions (\ref{om:2}), (\ref{om:23}) and
Proposition \ref{clm:1} to guarantee that the same arguments work. Thus
the proof of Theorem \ref{t:12} is finished.

\

\section{ Construction of the system $\Theta$}

As we have explained in Section \ref{ubos:1} we will obtain the
system $\Theta$ with the help of corresponding orthogonal
transformations. It is easy to check that dissolution process
will work if instead of  the matrix (\ref{m:1}) one takes any
$2\times2$ orthogonal matrix with elements
by modulus strictly less than one. The advantage of the matrix (\ref{m:1}) resides,
particularly,  on the fact that the elements of the first row of
the resulting matrix are equal. In fact if we apply the process
$2^n -1$ times then we will obtain the Haar matrix. But for our
purposes we do not need such details.

We define
\begin{equation}\label{ome:13}
\theta_{i}(x) = \chi_{0}^{(i)}(x), \quad  1 \leq i\leq \nu_0.
\end{equation}
Afterwards any block of functions $ \chi_{j}^{(i)}(x) (1 < i \leq
2^{n_j})$ will be transformed by the corresponding orthogonal matrix $H_{n_j}$.

We define
\begin{equation}\label{ome:14}
\theta_{j}^{(i)}(x) = \sum_{k=1}^{2^{n_j}} a_{ki}^{(j)}
\chi_{j}^{(k)}(x) \quad   1\leq i \leq 2^{n_j}; j\in \IN
 \end{equation}
(see \eqref{eq:hm1}) and denote
\begin{equation}\label{ome:15}
 \theta_{k}(x)=
 \theta_{j}^{(i)}(x), \qquad \mbox{where}\quad  k= \varpi_{j-1}  +i;
 \end{equation}
\[
 1 \leq i \leq 2^{n_{j}}, j\in \IN \quad \mbox{and}\quad
\varpi_{j} = \nu_0 + \sum_{l=1}^{j} 2^{n_l}.
\]
By
(\ref{ome:13})-(\ref{ome:15}) and the definition of the orthogonal
matrices $H_j$ we obtain that $\Theta = \{
\theta_k(x)\}_{k=1}^{\infty}$ is a complete ONS and (\ref{ub:1})
holds. Moreover, we also have  that for any collection of
coefficients $\{ c_i\}_{i=1}^{2^{n_j}}$ and any $j\in \IN$
\[
\sum_{i=1}^{2^{n_j}} c_i \theta_{j}^{(i)}(x) =
\sum_{i=1}^{2^{n_j}} b_i \chi_{j}^{(i)}(x),
\]
where
\[
\sum_{i=1}^{2^{n_j}} |c_i|^2 = \sum_{i=1}^{2^{n_j}} |b_i|^2.
\]
Hence from Theorem \ref{t:12}  follows that  $\Theta $ is a
divergence system. To show that the system $\Theta $ is a
convergence system we observe that the system
\[
 \theta_{j}^{(i)}(x) - 2^{-\frac{n_j}2} \chi_{j}^{(1)}(x)  \quad   1\leq i \leq 2^{n_j}; j\in \IN
\]
is an $S_p$ system for any $p>2.$ From (\ref{om:2}) we decompose
any function $\theta_k(x), k> \nu_0$ in the following form:
\[
\theta_{j}^{(i)}(x)=  2^{-\frac{n_j}2} \Upsilon_j(x) +
\big[\theta_{j}^{(i)}(x) - 2^{-\frac{n_j}2} \chi_{j}^{(1)}(x)\big]
\quad 1\leq i \leq 2^{n_j}; j\in \IN
\]
Hence by Theorem \ref{t:10} and Proposition \ref{pro:1} we easily
obtain that $\Theta$ is a convergence  system. Proof of Theorem
\ref{t:1} is finished.

\end{document}